\documentclass[smallextended]{article}       % onecolumn (second format)
\usepackage{graphicx}
\usepackage{subfigure}
\usepackage[T1]{fontenc}
\usepackage{amsthm}
\usepackage{graphicx}
\usepackage{amsfonts}
\usepackage{amsmath, xypic}
\newtheorem{thm}{Theorem}

\newtheorem{lem}{Lemma}
\newtheorem{prop}{Proposition}
\newtheorem{cor}{Corollary}
\newtheorem{exm}{Example}

\newtheorem{rem}{Remark}

\newcommand{\al}{\alpha}

\newcommand{\la}{\lambda}

\newlength{\cellsz}
\newcounter{cellsize}
\newcommand{\setcellsize}[1]{%
  \setcounter{cellsize}{#1}%
  \setlength{\cellsz}{\value{cellsize}\unitlength}}%
\setcellsize{12}%
\newcommand\cellify[1]{\def\thearg{#1}\def\nothing{}%
\hbox to 0pt{{\begin{picture}(\value{cellsize},\value{cellsize})
  \put(0,0){\line(1,0){\value{cellsize}}}
  \put(0,0){\line(0,1){\value{cellsize}}}
  \put(\value{cellsize},0){\line(0,1){\value{cellsize}}}
  \put(0,\value{cellsize}){\line(1,0){\value{cellsize}}} \end{picture}
\hss}}%\fi%
\vbox to \cellsz{ \vss \hbox to \cellsz{\hss$#1$\hss} \vss}}
\newcommand\tableau[1]{\vcenter{\vbox{\let\\\cr
\baselineskip -16000pt \lineskiplimit 16000pt \lineskip 0pt
\ialign{&\cellify{##}\cr#1\crcr}}}}
\newcommand\tabl[1]{\vtop{\let\\\cr
\baselineskip -16000pt \lineskiplimit 16000pt \lineskip 0pt
\ialign{&\cellify{##}\cr#1\crcr}}}
\newlength{\varcellsz}
\newcounter{varcellsize}
\newcommand{\setvarcellsize}[1]{%
  \setcounter{varcellsize}{#1}%
  \setlength{\varcellsz}{\value{varcellsize}\unitlength}}%
\setvarcellsize{10}%
\newcommand\varcellify[1]{\def\varthearg{#1}\def\varnothing{}%
\hbox to 0pt{{\begin{picture}(\value{varcellsize},\value{varcellsize})
  \put(0,0){\line(1,0){\value{varcellsize}}}
  \put(0,0){\line(0,1){\value{varcellsize}}}
  \put(\value{varcellsize},0){\line(0,1){\value{varcellsize}}}
  \put(0,\value{varcellsize}){\line(1,0){\value{varcellsize}}} \end{picture}
\hss}}%\fi%
\vbox to \varcellsz{ \vss \hbox to \varcellsz{\hss$#1$\hss} \vss}}
\newcommand\vartableau[1]{\vcenter{\vbox{\let\\\cr
\baselineskip -16000pt \lineskiplimit 16000pt \lineskip 0pt
\ialign{&\varcellify{##}\cr#1\crcr}}}}
\newcommand\vartabl[1]{\vtop{\let\\\cr
\baselineskip -16000pt \lineskiplimit 16000pt \lineskip 0pt
\ialign{&\varcellify{##}\cr#1\crcr}}}

\begin{document}

\title{Polynomial properties of Jack connection coefficients and generalization of a result by D\'enes}
%\titlerunning{Jack connection coefficients}
\author{Ekaterina A. Vassilieva}

%\authorrunning{Short form of author list} % if too long for running head

%\institute{Ekaterina A. Vassilieva \at
%              LIX, Ecole Polytechnique, Palaiseau, France \\
%              Tel.: +33177578039\\
%              \email{ekaterina.vassilieva@lix.polytechnique.fr}           
%}

%\date{Received: date / Accepted: date}
% The correct dates will be entered by the editor

\maketitle

\begin{abstract}

This article is devoted to the computation of Jack connection coefficients, a generalization of the connection coefficients of two classical commutative subalgebras of the group algebra of the symmetric group: the class algebra and the double coset algebra. The connection coefficients of these two algebraic structures are of significant interest in the study of Schur and zonal polynomials as well as the irreducible characters of the symmetric group and the zonal spherical functions. Furthermore they play an important role in combinatorics as they give the number of factorizations of a permutation into a product of permutations with given cyclic properties and, in some cases, the number of hypermaps embedded in orientable and locally orientable surfaces with specified vertex degree distribution.\\
Usually studied separately, these two families of coefficients share strong similar properties. First (partially) introduced by Goulden and Jackson in 1996, Jack connection coefficients provide a natural unified approach closely related to the theory of Jack polynomials, a family of bases in the ring of symmetric functions indexed by a parameter $\alpha$ that generalizes both Schur (case $\alpha = 1$) and zonal polynomials (case $\alpha = 2$). Jack connection coefficients are also directly linked to Jack characters, a general view of the characters of the symmetric group and the zonal spherical functions. Goulden and Jackson conjectured that these coefficients are polynomials in $\alpha$ with nice combinatorial properties, the so-called {\em Matchings-Jack conjecture}.\\
In this paper, we use the theory of Jack symmetric functions and the Laplace Beltrami operator to show the polynomial properties of Jack connection coefficients in some important cases. We also provide explicit formulations including notably a generalization of a classical formula of Denes for the number of minimal factorizations of a permutation into transpositions.    
\end{abstract}
%\keywords{Jack symmetric functions \and  Connection coefficients \and  Jack characters \and  Laplace Beltrami operator}
%%%%%%%%%%%%%%%%%%%%%%%%%%%%%%%%%%%%%%%%%%%%%%%%%%%%%%%%%%%%%%
\section{Introduction}
\label{sec:in}
%%%%%%%%%%%%%%%%%%%%%%%%%%%%%%%%%%%%%%%%%%%%%%%%%%%%%%%%%%%%%%
\subsection{Basic notations}
For any integer $n$ we note $S_n$ the symmetric group on $n$ elements and $\la=(\la_1,\la_2,\ldots,\la_p) \vdash n$ an integer partition of $|\la| = n $ with $\ell(\la)=p$ parts sorted in decreasing order. If $m_i(\la)$ is the number of parts of $\la$ that are equal to $i$, then we may write $\la$ as $[1^{m_1(\la)}\,2^{m_2(\la)}\ldots]$ and define $Aut_{\la}=\prod_i m_i(\la)!$ and $z_\lambda =\prod_i i^{m_i(\lambda)}m_i(\lambda)!$. A partition $\lambda$ is usually represented as a Young diagram of $|\la|$ boxes arranged in $\ell(\la)$ lines so that the $i$-th line contains $\la_i$ boxes.  Given a box $s$ in the diagram of $\lambda$, let $l'(s),l(s),a(s),a'(s)$ be the number of boxes to the north, south, east, west of $s$ respectively. These statistics are called {\bf co-leglength, leglength, armlength, co-armlength} respectively. We note for some parameter $\alpha$:
\begin{align}
h_{\lambda}(\alpha)= \prod_{s\in \lambda} (\alpha a(s) + l(s) + 1), \;\;\;\;\;\;\;  h'_{\lambda}(\alpha)= \prod_{s\in \lambda} (\alpha(1+a(s)) + l(s)).
\end{align}
Moreover, $\la'$ is the conjugate of partition $\la$, for two integer partitions $\la$ and $\mu$, we denote $\la>\mu$ if for all $i\geq 1$, $\la_1 + \la_2 +\ldots+\la_i >\mu_1 + \mu_2 +\ldots+\mu_i $ and we denote $n(\la)$ the quantity
\begin{equation}
\label{eq : n}
n(\la) = \sum_i (i-1)\la_i.
\end{equation}
%%%%%%%%%%%%%%%%%%%%%%%%%%%%%%%%%%%%%%%%%%%%%%%%%%%%%%%%%%%%%%
\subsection{Jack symmetric functions}
\indent Let $\Lambda$ be the ring of symmetric functions. Denote $m_\lambda(x)$ the monomial symmetric function indexed by $\lambda$ on indeterminate $x$, $p_\lambda(x)$ and $s_\lambda(x)$  the power sum and Schur symmetric functions respectively. Whenever the indeterminate is not relevant we shall simply write $m_\la$, $p_\la$ and $s_\la$. Let $\langle \cdot\, ,\cdot \rangle$ be the scalar product on $\Lambda$ such that the power sum symmetric functions verify $\left <p_\la,p_\mu \right> = z_\la\delta_{\la\mu}$ where $\delta_{\la\mu}$ is the Kronecker delta.\\
The Schur symmetric functions $s_\la$ are characterized by the fact that they form an orthogonal basis of $\Lambda$ for $\langle \cdot\, ,\cdot \rangle$ (they form even an orthonormal basis) and the transition matrix between Schur and monomial symmetric functions is upper unitriangular. As shown by Schur, the $s_\la$ are of critical interest in the representation theory of the symmetric group.\\
The {\bf zonal polynomials} $Z_\la$ constitute another important basis of $\Lambda$ directly linked with the theory of the {\bf zonal spherical functions}. Zonal polynomials verify the same properties as the $s_\la$ if the scalar product is replaced by $\langle \cdot,\cdot \rangle_2$ with $\langle p_\la,p_\mu \rangle_2 = 2^{\ell(\la)}z_\la\delta_{\la\mu}$.\\
More generaly, using an additional parameter $\alpha$, Henry Jack \cite{HJ} introduced the basis of {\bf Jack symmetric functions} $J^\alpha_\la$, that can be characterized as the set of symmetric functions verifying:
\begin{itemize}
\item[(a)] The $J^\alpha_\la$ are orthogonal for the alternative scalar product $\langle \cdot\, ,\cdot \rangle_\alpha$ that verifies:
\begin{equation}
\left<p_\la,p_\mu \right>_\alpha = \alpha^{\ell(\la)}z_\la\delta_{\la\mu}.
\end{equation}
\item[(b)] The transition matrix between the $J^\alpha_\la$ and the monomial symmetric functions is upper triangular and the coefficient in $m_\la$ of the expansion of $J^\alpha_\la$ in the monomial basis is equal to $h_\la(\alpha)$. Formally it means that the $J^\alpha_\la$ may be expressed with the help of some scalar coefficients $u_{\la\mu}^\alpha$ as: 
\begin{equation}
J^\alpha_\la = h_\la(\alpha)m_\la + \sum_{\mu < \la}u_{\la\mu}^\alpha m_\mu.
\end{equation}
\end{itemize}
According to the above definition, $J^1_\la$ is the normalized Schur symmetric function $h_\la(1)s_\la$ and $J^2_\la$ is the zonal polynomial $Z_\la$.
%%%%%%%%%%%%%%%%%%%%%%%%%%%%%%%%%%%%%%%%%%%%%%%%%%%%%%%%%%%%%%
\subsection{Jack connection coefficients -- main results}
In \cite{GJ96} Goulden and Jackson studied the series
\begin{equation}
\Phi(x,y,z,\alpha) = \sum_{\beta \vdash n}\frac{J^{\al}_\beta(x)J^{\al}_\beta(y)J^{\al}_\beta(z)}{\langle J^{\al}_\beta,J^{\al}_\beta\rangle_{\al}}.
\end{equation}
More precisely, they investigate the {\bf connection series} i.e. the coefficients $c_{\mu\nu}^\la(b)$ in the power sum expansion of $\Phi$ defined by
\begin{equation}
\sum_{\la, \mu, \nu \vdash n}\frac{c_{\mu\nu}^\la(b)}{(1+b)^{\ell(\la)}}\frac{|C_\la|}{n!}p_\la(x)p_\mu(y)p_\nu(z) = \sum_{\beta \vdash n}\frac{J^{1+b}_\beta(x)J^{1+b}_\beta(y)J^{1+b}_\beta(z)}{\langle J^{1+b}_\beta,J^{1+b}_\beta\rangle_{1+b}}
\end{equation} 
They conjecture that the $c_{\mu\nu}^\la(b)$ are polynomials in $b$ with positive integer coefficients and of degree at most $n-\min\{\ell(\mu),\ell(\nu)\}$. They also conjecture that the coefficients $c_{\mu\nu}^\la(b)$ can be written in the form
\begin{equation}
c_{\mu\nu}^\la(b) = \sum_{\delta}b^{wt_{\la}(\delta)}
\end{equation}
where the summation is over a subset of matchings (fixed point free involutions in $S_{2n}$) and $wt_{\la}$ is a combinatorial statistic correlated with the "nonbipartiteness" of matchings. They show the conjecture in the cases $\la =[1^n]$ and $\la = [1^{n-2}2^1]$.\\
As shown in Section \ref{sec : mot}, aside their importance in the study of Jack symmetric functions (and therefore Schur and zonal polynomials), there are several motivations to study this kind of coefficients. They provide a general approach to the study of classical connection coefficients in the group algebra of the symmetric group. These coefficients are also strongly related to the theory of Jack characters that generalize both the irreducible characters of the symmetric group and the zonal spherical functions. Finally (see e.g. \cite{H92}), they are of notable interest to statisticians as they arise in the study of asymmetric random walks on the symmetric group.\\  
In this paper we develop further this unified framework to study connection coefficients. We define the numbers $a_{\la^1,\la^2,\ldots,\la^s}$ for integer $s$ greater or equal to $1$ and $\la^i \vdash n$ for $1\leq i \leq s$.
\begin{equation}
\label{eq : df}
 a_{\la^1,\la^2,\ldots,\la^s}(\al) =\left [\prod_{i} p_{\la^i}(x^i)\right ]\sum_{\beta \vdash n}\frac{1}{\langle J^{\al}_\beta,J^{\al}_\beta\rangle_{\al}}\prod_{i}J_\beta^\alpha(x^i)
\end{equation}
We use the terminology of {\bf Jack connection coefficients} for these numbers. 
Notably we focus on the case when most of the $\la^i$ are equal to $\rho = [1^{n-2}2^1]$. We denote $a^r_\la(\al) =  a_{\la, \rho,\ldots,\rho}(\al)$ (with $\rho$ appearing $r$ times, $r\geq0$), i.e.: 
\begin{equation}
\label{eq : ar}
a^r_\la(\al) =\left [p_{\la}(x^1)\prod_{i \geq 2}p_{\rho}(x^i)\right ] \sum_{\beta \vdash n}\frac{1}{\langle J^{\al}_\beta,J^{\al}_\beta\rangle_{\al}}\prod_{i}J_\beta^\alpha(x^i)
\end{equation}
The coefficients $a^r_\la(\al)$ are of particular interest. As shown by Hanlon in \cite{H92}, they give the probability that a random walk of length $r$ on $S_n$ which begins at the identity permutation ends at a permutation of type $\lambda$. The transition probabilities for this random walk are such that only moves from a permutation $\pi$ to $\pi\circ\tau$ where $\tau$ is a transposition are possible. The probability of moving from a permutation of cycle length $l$ to a permutation of cycle length $l-1$ is $\alpha$ times greater than moving to a permutation of cycle length $l+1$. Additionally, as shown in Section \ref{sec : JCCLBO}, the $a^r_\la(\al)$ give the coefficients in the power sum expansion of the symmetric function obtained by applying $r$ times the Laplace Beltrami operator to $p_1^n$.   
\begin{rem}
The motivations for choosing the definition of Equation (\ref{eq : df}) for Jack connection coefficients are detailed in Section \ref{sec : cccadca}, Equations (\ref{eq : lemSymJC}) and (\ref{eq : lemdcaJC}).
\end{rem}
Denote $r_\la = |\la|-\ell(\la)$. In Section \ref{sec : PT}, we show the following main results.
\begin{thm}
\label{thm : zero}
Let $a^r_\la(\al)$  be defined as above. Then for any integer partition $\la$ we have $a^r_\la(\al) = 0$ for $r<r_\la$.
\end{thm}
\begin{thm}
\label{thm : pol}
For $r \geq r_\la$, $a^r_\la(\al)$ can be written in the form
\begin{equation}
a^r_\la(\al) = \frac{1}{n!}\sum_{i=0}^{r-r_\la}g_{\la,r}^i\al^{i-\ell(\la)}
\end{equation}
where the $g_{\la,r}^i$ are integers that verify 
\begin{equation}
g_{\la,r}^i = (-1)^{r-r_\la}g_{\la,r}^{r-r_\la-i}.
\end{equation}
\end{thm}
As a consequence of Theorem \ref{thm : pol} we have the following corollary:
\begin{cor}
Let $\la$ be an integer partition of $n$ and $r$ an integer greater than $r_\la$, then $n!\al^{\ell(\la)}a^r_\la(\al)$ is a polynomial in $\al$ with integer coefficients of degree at most $r-r_\la$.
 \end{cor}
 We also show:
\begin{thm}
\label{thm : polpos}
Let $\la$ be an integer partition of $n$ and $r$ an integer greater than $r_\la$, then $n!\al^{\ell(\la)}a^r_\la(\al)$ is a polynomial in $b = \al-1$ with {\bf non negative} integer coefficients of degree at most $r-r_\la$.
 \end{thm}
\begin{thm}
\label{thm : denes}
Let $a^r_\la(\al)$  be defined as above. In the limit case $r = r_\la$, we have
\begin{equation}
\label{eq : thm}
a^{|\la|-\ell(\la)}_\la(\al) = \frac{(|\la|-\ell(\la))!}{\al^{\ell(\la)}Aut_\la\prod_i\la_i! }\prod_i\la_i^{\la_i-2}.
\end{equation}
 \end{thm}
\begin{rem}
In the specific case $\la = (n)$, Equation~\ref{eq : thm} reads
\begin{equation}
\label{eq : gendenes}
a^{n-1}_{(n)}(\al) = \frac{1}{\al n}n^{n-2}.
\end{equation}
We view this later formula as a generalization of the classical formula attributed to D\'enes {\em \cite{D59}} for the number of minimal factorizations of a long cycle in the symmetric group into a product of transpositions.
\end{rem}
\subsection{Organization of the paper}
The paper is organized as follows. In Section \ref{sec : mot} we focus on the motivations for studying Jack connection coefficients as defined in Equation (\ref{eq : df}). We also provide some background results about classical connections coefficients. Section \ref{sec : PT} provides the proofs of Theorems \ref{thm : zero}, \ref{thm : pol}, \ref{thm : polpos} and \ref{thm : denes}. Section \ref{sec : FR} gives additional results based on the development in the previous section. 
%%%%%%%%%%%%%%%%%%%%%%%%%%%%%%%%%%%%%%%%%%%%%%%%%%%%%%%%%%%%%%
\section{Motivations and background}
\label{sec : mot}
\subsection{Jack characters}
For any integer partition $\la$ of $n$ let $C_\lambda$ be the {\bf conjugacy class} of $S_n$ containing the permutations of cycle type $\lambda$. The cardinality of the conjugacy classes is given by $|C_\la| = n!/z_\la$. Additionally, $B_n$ is the {\bf hyperoctahedral group} (i.e. the centralizer of $f_\star = (12)(34)\ldots(2n-1\,2n)$ in $S_{2n}$). As shown in e.g. \cite[VII.2]{IGM} the {\bf double cosets} of $B_n$ in $S_{2n}$ are also indexed by integer partitions of $n$. We denote by $K_\lambda$ the double coset indexed by $\la \vdash n$ consisting of all the permutations $\omega$ of $S_{2n}$ such that $f_\star \circ\omega\circ f_\star\circ\omega^{-1}$ has cycle type $(\lambda_1,\lambda_1, \lambda_{2},\lambda_{2},...,\lambda_p,\lambda_p)$. We have $|B_n| = 2^nn!$ and $|K_\la| = |B_n|^2/(2^{\ell(\la)}z_\la)$.
The link between Schur (resp. zonal) polynomials and irreducible characters of the symmetric group (resp. zonal spherical functions) is given by the decomposition of the $s_\la$ (resp. $Z_\la$) in the power sum basis
\begin{align}
\label{eq : characters}s_\la &= \sum_{\mu \vdash n}z_\mu^{-1}\chi^\la_\mu p_\mu,\\
\label{eq : spherical}Z_\la &= \frac{1}{| B_n |}\sum_{\mu \vdash n}\varphi^\beta_\mu p_\mu,
\end{align}
where $\chi^\la_\mu$ is the value of the irreducible character of the symmetric group $\chi^\la$ indexed by integer partition $\la$ at any element of $C_\mu$ and $\varphi^\beta_\mu = \sum_{\omega \in K_\mu}\chi^{2\beta}(\omega)$. The value of the zonal spherical function indexed by $\lambda$ of the Gelfand pair $(S_{2n}, B_n)$ at the elements of the double coset $K_\mu$ is given by $| K_\mu |^{-1} \varphi^{\la}_\mu$.
Using the fact that both Schur and zonal symmetric functions are special cases of Jack symmetric functions, it is natural to focus on a more general form of Equations (\ref{eq : characters}) and (\ref{eq : spherical}). Formally, let $\theta_\mu^\la(\alpha)$ denote the coefficient of $p_\mu$ in the expansion of $J^\alpha_\la$ in the power sum basis:
\begin{equation}
J^\alpha_\la = \sum_{\mu}\theta_\mu^\la(\alpha)p_\mu.
\end{equation}
According to Equations (\ref{eq : characters}) and (\ref{eq : spherical}), up to a normalization factor, the $\theta_\mu^\la(\alpha)$'s coincide with the irreducible characters of the symmetric group and the zonal spherical functions in the cases $\alpha =1$ and $\alpha = 2$. In the general case Do\l \k{e}ga and F\'eray \cite{DF} named them {\bf Jack characters}.\\
The coefficients in the power sum expansion of Jack symmetric functions have received significant attention over the past decades. As an example, Hanlon conjectured a first combinatorial interpretation for them in \cite{H88} in terms of digraphs. Stanley proved various results for these coefficients in \cite{S89}. For instance for simple values of $\mu$ we have the following formulas:
\begin{align}
&\theta_{[1^n]}^\la(\alpha) = 1,\\
&\theta_{[1^{n-2}2^1]}^\la(\alpha) = \alpha n(\la') - n(\la),\\
&\theta_{(n)}^\la(\alpha) = \prod_{s \in \la \setminus\{(1,1)\}}(\al a'(s)-l'(s)).
\end{align}
More recent works by Lasalle (see e.g. \cite{ML} and \cite{ML09}) reconsidered these coefficients as generalizations of irreducible characters of the symmetric group. Lasalle conjectured various polynomial properties in $\al$ for the $\theta_\mu^\la(\alpha)$. Do\l \k{e}ga and F\'eray (\cite{DF}) proved them partially and showed that Jack characters are polynomials in $\al$ with rational coefficients. Together with P. \'Sniady \cite{DFS} they conjectured an expression involving a measure of "non-orientability" of  locally orientable hypermaps (an expression also introduced in \cite{GJ96}). Do\l \k{e}ga and F\'eray also proved in \cite{DF14} that the $c^\la_{\mu\nu}(b)$ introduced by Goulden and Jackson are polynomials in $b$ with rational coefficients.  
%%%%%%%%%%%%%%%%%%%%%%%%%%%%%%%%%%%%%%%%%%%%%%%%%%%%%%%%%%%%%%
\subsection{Connection coefficients of the class algebra and the double coset algebra}
\label{sec : cccadca}
By abuse of notation let $C_\lambda$ (resp. $K_\lambda$) also represent the formal sum of its elements in the group algebra $\mathbb{C} S_{n}$ (resp. $\mathbb{C} S_{2n}$). Then $\{C_\lambda, \lambda \vdash n\}$ (resp. $\{K_\lambda, \lambda \vdash n\}$) forms a basis of the {\bf class algebra} (resp.  {\bf double coset algebra}, i.e. the commutative subalgebra of $\mathbb{C} S_{2n}$ identified as the Hecke algebra of the Gelfand pair $(S_{2n},B_n)$).
For $\lambda^{i} \vdash n$ ($1 \leq i\leq s$), we define the {\bf connection coefficients of the class algebra} $
c^{\la^1}_{\lambda^2,\ldots,\la^s}$ and the {\bf connection coefficients of the double coset algebra} $b^{\la^1}_{\lambda^2,\ldots,\la^s}$ by
\begin{equation}
c^{\la^1}_{\lambda^2,\ldots,\la^s} = [C_{\la^1}]\prod_{i \geq 2}C_{\lambda^i}, \;\;\;\;\; b^{\la^1}_{\lambda^2,\ldots,\la^s} = [K_{\la^1}]\prod_{i \geq 2}K_{\lambda^i}.
\end{equation}
From a combinatorial point of view $c^{\la^1}_{\lambda^2,\ldots,\la^s}$ is the number of ways to write a given permutation $\sigma_1$ of $C_{\la^1}$ as the ordered product of $s-1$ permutations $\sigma_2\circ\ldots\circ\sigma_s$ where $\sigma_i$ is in $C_{\lambda^i}$. Similarly, $b^{\la^1}_{\lambda^2,\ldots,\la^s}$ counts the number of ordered factorizations of a given element in $K_{\la^1}$ into $s-1$ permutations of $K_{\lambda^2},\ldots,K_{\lambda^s}$. Connections coefficients are also strongly related to characters and zonal spherical functions.
\begin{lem}\label{lem : sym} The connection coefficients  of the symmetric group and the double coset algebra are linked to the irreducible characters of the symmetric group and the zonal spherical functions through the equations
\begin{align}
\label{eq : lemSym}c^{\la^1}_{\lambda^2,\ldots,\la^s} &=\frac{n!}{|C_{\la^1}|}\sum_{\beta \vdash n}\frac{1}{h_{\beta}(1)^2}\prod_{i \geq 1}h_{\beta}(1)z_{\la^i}^{-1}\chi^\beta_{\la^i}, \\
\label{eq : lemdca}b^{\la^1}_{\lambda^2,\ldots,\la^s} &= \frac{1}{|K_{\la^1}|}\sum_{\beta \vdash n}\frac{1}{h_\beta(2)h'_\beta(2)}\prod_{i \geq 1}\varphi^\beta_{\la^i}.
\end{align}
\end{lem}
\begin{proof}[Equation~\ref{eq : lemSym}]
Using the fact that the elements
\begin{equation}
E_\la = \frac{deg(\chi^\la)}{|S_n|}\sum_{\mu \vdash n}\chi^\la_\mu C_\mu
\end{equation}
form a complete set of orthogonal indempotents for the center of $\mathbb{C} S_{n}$, it is easy to show (see \cite{HSS} for the detailed computation) that
\begin{align}
\prod_{i \geq 2}C_{\lambda^i}  = \frac{1}{|S_n|}\sum_{\mu \vdash n}C_\mu\sum_{\beta \vdash n}\frac{\prod_{i \geq 2}|C_{\lambda^i}|\chi^\beta_{\la^i}}{deg(\chi^\beta)^{s-2}}\chi^\beta_\mu.
\end{align}
As a result we get
\begin{equation}
c^{\la^1}_{\lambda^2,\ldots,\la^s} = [C_{\la^1}]\prod_{i \geq 2}C_{\lambda^i} = \frac{1}{|S_n|}\sum_{\beta \vdash n}\frac{\prod_{i \geq 2}|C_{\lambda^i}|\chi^\beta_{\la^i}}{deg(\chi^\beta)^{s-2}}\chi^\beta_{\la^1}.
\end{equation}
But $deg(\chi^\beta) = n!/h_\beta(1)$ and $|C_{\lambda}| = n!/z_\la$. This provides the desired result. 

%using the power sum expansion of Shur symmetric functions we have 
%\begin{align}
%\nonumber J^1_\la = h_\la(1)s_\la &=h_\la(1) \sum_{\mu \vdash n}z_\mu^{-1}\chi^\la_\mu p_\mu\\ 
%\nonumber &=\frac{h_\la(1)}{n!} \sum_{\mu \vdash n}|C_{\mu}|\chi^\la_\mu p_\mu\\
%&= \sum_{\mu \vdash n}\frac{|C_{\mu}|\chi^\la_\mu}{deg(\chi^\la)} p_\mu
%\end{align}
%According to the definition of Jack characters this leads to:
%\begin{align}
%\nonumber c^{\la^1}_{\lambda^2,\ldots,\la^r} &= \frac{1}{n!|C_{\la^1}|}\sum_{\beta \vdash n}deg(\chi^\beta)^2\prod_{i \geq 2}\theta^\beta_{\la^i}(1)\frac{|C_{\la^1}|\chi^\beta_{\la^1}}{deg(\chi^\beta)}\\
%&=\frac{n!}{|C_{\la^1}|}\sum_{\beta \vdash n}\frac{1}{{h_{\beta}(1)}^2}\prod_{i \geq 1}\theta^\beta_{\la^i}(1)
%\end{align}
%Finally, noticing that $h_{\beta}(1) = h'_{\beta}(1)$ and recalling that $n!/|C_{\la^1}| = z_{\la^1}$, we get the desired result. 
\end{proof}
\begin{proof}[Equation~\ref{eq : lemdca}]
Using similar techniques as in the proof of Equation (\ref{eq : lemSym}), one can show (see~\cite{HSS}) that
\begin{equation}
\prod_{i \geq 2}K_{\lambda^i}  = \sum_{\mu \vdash n}K_\mu\frac{1}{|K_\mu|}\sum_{\beta \vdash n}\frac{1}{h_\beta(2)h'_\beta(2)}\varphi^\beta_\mu\prod_{i \geq 2}\varphi^\beta_{\la^i}.
\end{equation}
Extracting the coefficient in $K_{\la^1}$ gives the desired formula.
%where $\varphi^\beta_\mu = \sum_{\omega \in K_\mu}\chi^{2\beta}(\omega)$. But 
%\begin{equation}
%J^2_\beta = Z_\beta = \frac{1}{|B_n|}\sum_{\mu \vdash n}\varphi^\beta_\mu p_\mu
%\end{equation}
%Therefore we obtain
%\begin{equation}
%b^{\la^1}_{\lambda^2,\ldots,\la^r} = \frac{|B_n|^{r}}{|K_{\la^1}|}\sum_{\beta \vdash n}\frac{1}{h_\beta(2)h'_\beta(2)}\prod_{i \geq 1}\theta^\beta_{\la^i}(2)
%\end{equation}
%Using $|K_{\la^1}| = |B_n|^2/(2^{\ell(\la^1)}z_{\la^1})$ yields the desired relation.
\end{proof}
Using the formalism of Jack characters Equations (\ref{eq : lemSym}) and (\ref{eq : lemdca}) read:
\begin{align}
\label{eq : lemSymJC}c^{\la^1}_{\lambda^2,\ldots,\la^s} &=\frac{n!}{|C_{\la^1}|}\sum_{\beta \vdash n}\frac{1}{{h_{\beta}(1)}{h'_{\beta}(1)}}\prod_{i \geq 1}\theta^\beta_{\la^i}(1),\\
\label{eq : lemdcaJC}b^{\la^1}_{\lambda^2,\ldots,\la^s} &= \frac{|B_n|^{s}}{|K_{\la^1}|}\sum_{\beta \vdash n}\frac{1}{h_\beta(2)h'_\beta(2)}\prod_{i \geq 1}\theta^\beta_{\la^i}(2).
\end{align}
Equations (\ref{eq : lemSymJC}) and (\ref{eq : lemdcaJC}) are very similar and the question of a more general approach to connection coefficients for any parameter $\alpha$ appears to be very natural. Jack connection coefficients provide the desired approach. For the sake of simplicity, further in this paper we use the following formulation for Jack connection coefficients:
\begin{equation}
 a_{\la^1,\la^2,\ldots,\la^s}(\al) = \sum_{\beta \vdash n}\frac{1}{h_\beta(\alpha)h'_\beta(\alpha)}\prod_{i}\theta^\beta_{\la^i}(\alpha),
\end{equation}
\begin{equation}
a^r_\la(\al) = \sum_{\beta \vdash n}\frac{1}{h_\beta(\alpha)h'_\beta(\alpha)}\theta^\beta_{\la}(\alpha)\left(\theta^\beta_{[1^{n-2}2^1]}(\alpha)\right)^r.
\end{equation}
This formulation is justified as Stanley showed in \cite{S89} that  the scalar product of the $J_\beta^\al$ may be written as:
\begin{equation}
{\langle J_\beta^\al,J_\beta^\al \rangle}_\al = h_\beta(\al)h'_\beta(\al).
\end{equation}

 %%%%%%%%%%%%%%%%%%%%%%%%%%%%%%%%%%%%%%%%%%%%%%%%%%%
\subsection{Background on connection coefficients}
Except for special cases no closed formulas are known for the coefficients  $c^{\la^1}_{\lambda^2,\ldots,\la^s}$ and $b^{\la^1}_{\lambda^2,\ldots,\la^s}$. Using an inductive argument B\'{e}dard and Goupil  \cite{BG} first found a formula for $c^n_{\lambda,\mu}$ in the case $\ell(\la)+\ell(\mu)=n+1$, which was later reproved by Goulden and Jackson \cite{GJ92} via a bijection with a set of ordered rooted bicolored trees. Later, using characters of the symmetric group and a combinatorial development, Goupil and Schaeffer \cite{GS} derived an expression for connection coefficients of arbitrary genus as a sum of positive terms (see Biane \cite{PB} for a succinct algebraic derivation; and Poulalhon and Schaeffer \cite{PS}, and Irving \cite{JI} for further generalizations). 
Closed form formulas of the expansion of the generating series for the $c^{(n)}_{\lambda^2,\ldots,\la^s}$ (for general $s$) and $b^{(n)}_{\lambda,\mu}$ in the monomial basis were provided by Morales and Vassilieva and Vassilieva in \cite{MV09}, \cite{MV11}, \cite{V2012} and \cite{V2013}. Papers \cite{MV11} and \cite{V2013} use the topological interpretation of $c^{(n)}_{\lambda^2,\ldots,\la^s}$ and $b^{(n)}_{\lambda,\mu}$ in terms of unicellular locally orientable hypermaps and constellations of given vertex degree distribution. Jackson (\cite{DMJ}) computed a general expression for the generating series of the $\sum_{\ell(\la_i) = p_i}c^{\la^1}_{\lambda^2,\ldots,\la^s}$ in terms of some explicit polynomials. This expression allows him to compute the following formula when $\la_1 = (n)$ and $\la_2 = \ldots =\la_s = [1^{n-2}2^1]$:
\begin{equation}
\label{eq : J}
c^{(n)}_{[1^{n-2}2^1],\ldots,[1^{n-2}2^1]} = \frac{(s-1)!}{n!}n^{s-1}2^{n-s}[X^{s-1}]sh^{n-1}X.
\end{equation}
When $s=n$, $c^{(n)}_{[1^{n-2}2^1],\ldots,[1^{n-2}2^1]}$ is equal to $n^{n-2}$. This is the famous result classically attributed to D\'enes (\cite{D59}). Shapiro, Shapiro and Vainhstein \cite{SSV} reproved the elegant closed form generating series in Equation (\ref{eq : J}). \\
Additionally we define 
\begin{equation}
d_\la = \left[p_\la\frac{u^{n+\ell(\la)-2}t^{|\la|}}{r(\la)!|\la|!}\right]\log\left(1+\sum_{\rho,k}\frac{t^{|\rho|}u^k}{z_\rho k!}c^{\rho}_{\underbrace{[1^{n-2}2^1],\ldots,[1^{n-2}2^1]}_{k\; factors}}p_\rho\right)
\end{equation}
where the sums over $\rho$ is over all the non empty partitions. Goulden and Jackson proved the following result (see \cite{GJ97}):
\begin{equation}
d_\la = n^{\ell(\la)-3}(n+\ell(\la)-2)! \prod_{i}\frac{\la_i^{\la_i}}{(\la_i-1)!}.
\end{equation}
This is the number of minimal transitive factorizations of a permutation in $C_\la$ into a product of transpositions.
%%%%%%%%%%%%%%%%%%%%%%%%%%%%%%%%%%%%%%%%%%%%%%%%%%%%%%%%%%%%%%
\section{Computation of Jack connection coefficients}
\label{sec : PT}
In this section we show various formulas for Jack connection coefficients that allow us to prove Theorems \ref{thm : zero}, \ref{thm : pol}, \ref{thm : polpos} and \ref{thm : denes}. 
\subsection{Inverting $\alpha$}
As a first result we show the relation between the values of Jack connection coefficients in parameters $\alpha$ and $\alpha^{-1}$. We have the following theorem:
\begin{thm} 
\label{thm : inv} Let $\la^i \vdash n$ for $1\leq i \leq s$ and $\al \neq 0$. The Jack connection coefficients for parameters $\al$  et $\alpha^{-1}$ are linked through the relation
\begin{equation}
a_{\la^1, \ldots, \la^s}(\al^{-1}) = (-\al)^{(2-s)n+\sum_{i}\ell(\la^i)}a_{\la^1, \ldots, \la^s}(\al).
\end{equation}
\end{thm}
\begin{proof} Let $\omega_\al$ be the automorphism on $\Lambda$ defined by 
\begin{equation}
\omega_\al p_r = -(-\al)^rp_r
\end{equation}
for any integer $r$.
We have (see \cite{IGM})
\begin{equation}
\omega_\al J_\beta^\al = \al^{|\la|}J_{\beta'}^{1/ \al}.
\end{equation}
As a consequence 
\begin{equation}
\theta^{\beta}_\la(\al^{-1}) = (-\al)^{\ell(\la)-|\la|}\theta^{\beta'}_\la(\al).
\end{equation}
Furthermore it is clear from the definition that 
\begin{equation}
h_\beta(\al^{-1})h'_\beta(\al^{-1}) =\al^{-2|\beta|}h_{\beta'}(\al)h'_{\beta'}(\al).
\end{equation}
Finally,
\begin{align}
a_{\la^1, \ldots, \la^s}(\al^{-1}) &=  \sum_{\beta \vdash n}\frac{\prod_{i}\theta^\beta_{\la^i}(\alpha^{-1})}{h_\beta(\alpha^{-1})h'_\beta(\alpha^{-1})}\\
&= (-\al)^{2n-sn+\sum_{i}\ell(\la^i)}\sum_{\beta \vdash n}\frac{\prod_{i}\theta^{\beta'}_{\la^i}(\alpha)}{h_{\beta'}(\alpha)h'_{\beta'}(\alpha)}
\end{align}
and the result follows.\end{proof}
\subsection{Jack connection coefficients and Laplace Beltrami operator}
\label{sec : JCCLBO}
For indeterminate $x = (x_1,x_2,\ldots)$ define the {\bf Laplace Beltrami Operator} by
\begin{equation}
D(\al) = \frac{\al}{2}\sum_{i}x_i^2\frac{\partial^2}{\partial x_i^2}+\sum_{i}\sum_{j\neq i}\frac{x_ix_j}{x_i-x_j}\frac{\partial}{\partial x_i}
\end{equation}
In this section we show that the coefficients in the power sum expansion of $D(\al)^r(p_1^n)$ are Jack connection coefficients. We have the theorem:
\begin{thm}\label{thm : ar} Let $a^r_\la(\al)$ be the Jack connection coefficients defined by Equation (\ref{eq : ar}). We have the following equality:
\begin{equation}
\label{eq : arD} a^r_\la(\al) =\frac{1}{\al^nn!} [p_\la]{D(\al)}^r(p_1^n),
\end{equation}
where $[p_\la]{D(\al)}^r(p_1^n)$ denotes the coefficient of $p_\la$ in the power sum expansion of ${D(\al)}^r(p_1^n)$.
\end{thm}
To prove Theorem \ref{thm : ar}, we use two classical properties of Jack symmetric functions. First, Stanley in \cite{S89} showed that  the scalar product of the $J_\la^\al$ may be written as
\begin{equation}
{\langle J_\la^\al,J_\mu^\al \rangle}_\al = \delta_{\la\mu}h_\la(\al)h'_\la(\al).
\end{equation}
As a consequence, we have the following identity for Jack characters:
\begin{equation}
\sum_{\rho}z_\rho\al^{\ell(\rho)}\theta_\rho^\la(\al)\theta_\rho^\mu(\al) = \delta_{\la\mu}h_\la(\al)h'_\la(\al).
\end{equation}
\noindent Equivalently, 
\begin{equation}
\label{eq : orth}
\sum_{\la}\frac{1}{h_\la(\al)h'_\la(\al)}\theta_\rho^\la(\al)\theta_\sigma^\la(\al) = \frac{\delta_{\rho\sigma}}{z_\rho\al^{\ell(\rho)}}.
\end{equation}
\noindent As a result we have the following lemma:
\begin{lem} The following relation between Jack and power sum symmetric functions holds:
\label{lem : jp}
\begin{equation}
\sum_{\la}\frac{1}{h_\la(\al)h'_\la(\al)}J^\al_\la = \frac{p_1^n}{\al^nn!}.
\end{equation}
\end{lem}
\begin{proof} It is a direct consequence from Equation (\ref{eq : orth}). Summing over $p_\sigma$ on both sides reads
\begin{align}
\nonumber \sum_{\la, \sigma}\frac{1}{h_\la(\al)h'_\la(\al)}\theta_\rho^\la(\al)\theta_\sigma^\la(\al)p_\sigma &= \frac{p_\rho}{z_\rho\al^{\ell(\rho)}},\\
\sum_{\la}\frac{1}{h_\la(\al)h'_\la(\al)}\theta_\rho^\la(\al)J_\la^\al &= \frac{p_\rho}{z_\rho\al^{\ell(\rho)}}.
\end{align}
The formula is proved by setting $\rho = [1^n]$. In this case $\theta_{[1^n]}^\la(\al) =1$ and the result follows.
\end{proof}
\noindent Secondly, it is well known that the Jack symmetric functions are eigenfunctions of $D(\al)$:
\begin{equation}
\label{eq : eigen}
D(\al)J_\la^\al = \left(\al n(\la') - n(\la)\right)J_\la^\al = \theta_{[1^{n-2}2^1]}^\la(\al)J_\la^\al.
\end{equation}
Theorem \ref{thm : ar} is a direct consequence of Lemma \ref{lem : jp} and Equation (\ref{eq : eigen}):
\begin{align*}
\sum_{\la \vdash n} a^r_\la(\al)p_\la &= \sum_{\la, \beta \vdash n}\frac{1}{h_\beta(\alpha)h'_\beta(\alpha)}\left(\theta^\beta_{{[1^{n-2}2^1]}}(\alpha)\right)^r\theta^\beta_{\la}(\alpha)p_\la\\
& = \sum_{\beta \vdash n}\frac{1}{h_\beta(\alpha)h'_\beta(\alpha)}\left(\theta^\beta_{{[1^{n-2}2^1]}}(\alpha)\right)^rJ_\beta^\al\\
& = \sum_{\beta \vdash n}\frac{1}{h_\beta(\alpha)h'_\beta(\alpha)}{D(\al)}^r(J_\beta^\al)\\
& = {D(\al)}^r \left (\sum_{\beta \vdash n}\frac{1}{h_\beta(\alpha)h'_\beta(\alpha)}J_\beta^\al \right)\\
& = {D(\al)}^r\left (\frac{p_1^n}{\al^nn!}\right).
\end{align*}
The second equality comes from the definition of $\theta^\beta_{\la}(\alpha)$ as the coefficients in the power sum expansion of the Jack symmetric functions, the third one is $r$ times the application of Equation (\ref{eq : eigen}) and the last one is the application of Lemma \ref{lem : jp}.

\begin{exm} Successive applications of the Laplace Beltrami operator to\linebreak $p_1^n/\al^nn!$ give the values of the $a^r_\la(\al)$'s. For $r=1$ or $2$ and $\la \vdash n$ we get
\begin{align}
{D(\al)}\left (\frac{p_1^n}{\al^nn!}\right) &= \frac{1}{2\al^{n-1}(n-2)!}p_{[1^{n-2}2^1]},\\
\nonumber {D(\al)}^2\left (\frac{p_1^n}{\al^nn!}\right) &=\frac{1}{2\al^{n-1}(n-2)!}p_{[1^n]}+ \frac{(\al-1)}{\al^{n-1}(n-2)!}p_{[1^{n-2}2^1]}\\
&\phantom{lalal}+\frac{1}{4\al^{n-2}(n-4)!}p_{[1^{n-4}2^2]}+\frac{1}{\al^{n-2}(n-3)!}p_{[1^{n-3}3^1]}.
\end{align}
\end{exm}
%%%%%%%%%%%%%%%%%%%%%%%%%%%%%%%%%%%%%%%%%%%%%%%%%%%%%%%%%%%
\subsection{Proof of Theorems \ref{thm : pol} and \ref{thm : polpos}}
In order to show Theorems \ref{thm : pol} and \ref{thm : polpos} we need an additional classical lemma:
\begin{lem}
\label{lem : dp} Operator $D(\al)$ can be expressed as
\begin{equation}
D(\al) = (\al-1)N + \al U + S,
\end{equation}
where for $\la \vdash n$:
\begin{align}
N(p_\la) &=  n(\la')p_\la,\\
\label{eq : defU} U(p_\la) &= \left (\frac{1}{2}\sum_{i,j}ijp_{i+j}\frac{\partial}{\partial p_i}\frac{\partial}{\partial p_j}\right )p_\la,\\
S(p_\la) &= \left (\frac{1}{2}\sum_{i,j}(i+j)p_ip_j\frac{\partial}{\partial p_{i+j}}\right )p_\la.
\end{align}
\end{lem}
The following property holds:
\begin{prop}
For any $\la \vdash n$, the coefficients in the power sum expansion of $N(p_\la)$, $U(p_\la)$ and $S(p_\la)$ are non negative integers.
\end{prop}
\begin{proof}:
\begin{itemize}
\item[(i)] From Equation (\ref{eq : n}), it is obvious that the quantity $n(\la')$ is a non negative integer so $[p_\la]N(p_\la)$ is a non negative integer.
\item[(ii)] The expression of $U(p_\la)$ is symmetric in $i$ and $j$ so the coefficient $1/2$ vanishes by adding the terms for indices $(i,j)$ and $(j,i)$ when $i\neq j$. When $i=j$, assume $m_i(\la) = k \geq 2$. Then $$\frac{\partial^2}{\partial p_i^2}p_\la = k(k-1)p_\la/p_i^{2}.$$ As $k(k-1)/2$ is a non negative integer, the coefficients in the power sum expansion of $U(p_\la)$ are always non negative integers.
\item[(iii)] Regarding operator $S$ the property is proved by noticing that for $i=j$, $(i+j)/2 = i$ is a non negative integer.
\end{itemize}
\end{proof}
Using Theorem \ref{thm : ar} we may rewrite Jack connection coefficients $a_\la^r(\al)$ as:
\begin{equation}
a_\la^r(\al) = \frac{1}{\al^nn!}{\left(\al (N+U) + (S-N)\right)}^r(p_1^n)
\end{equation}
As a result, $n!\al^na_\la^r(\al)$ is a polynomial in $\al$ of degree at most $r$ with integer coefficients. For $0\leq i \leq r$, there exist integers $h_{\la,r}^i$ such that 
\begin{equation}
a_\la^r(\al) = \frac{1}{\al^nn!}\sum_{0\leq i \leq r}h_{\la,r}^i\alpha^i.
\end{equation}
But from Theorem \ref{thm : inv}
\begin{equation}
a_\la^r(\al^{-1}) = (-\al)^{(1-r)n+r(n-1)+\ell(\la)}a_\la^r(\al) = (-\al)^{n-r+\ell(\la)}a_\la^r(\al).
\end{equation}
Recall the definition $r_\la = n-\ell(\la)$. As a consequence to the previous equation
\begin{align}
\al^n\sum_{0\leq i \leq r}h_{\la,r}^i\alpha^{-i} &= (-\al)^{n-r+\ell(\la)}\frac{1}{\al^n}\sum_{0\leq i \leq r}h_{\la,r}^i\alpha^i,\\
\sum_{0\leq i \leq r}h_{\la,r}^i\alpha^{n-i} &= (-1)^{r-r_\la}\sum_{0\leq i \leq r}h_{\la,r}^i\alpha^{i-r+\ell(\la)}.
\end{align}
Equating the coefficients in the powers of $\alpha$ in both sides of the equation finishes the proof of Theorem \ref{thm : pol}. We get
\begin{equation}
\label{eq : aaa}h_{\la,r}^{r+r_\la-i} = (-1)^{r-r_\la}h_{\la,r}^i.
\end{equation}
Using the previous equation, as the summation parameter $i$ is less or equal to $r$:
\begin{equation}
h_{\la,r}^j = 0 \mbox{ if } j<r_\la.
\end{equation}
Finally the following expressions hold
\begin{align}
a_\la^r(\al) &= \frac{1}{\al^nn!}\sum_{r_\la \leq i \leq r}h_{\la,r}^i\alpha^i\\
\label{eq : bbb}&= \frac{1}{n!}\sum_{0 \leq i \leq r-r_\la}h_{\la,r}^{i+r_\la}\alpha^{i-\ell(\la)}.
\end{align}
Equations (\ref{eq : aaa}) and (\ref{eq : bbb}) prove Theorem \ref{thm : pol}.\\
The proof of Theorem \ref{thm : polpos} is a direct consequence of
\begin{equation}
a_\la^r(\al) = \frac{1}{\al^nn!}{\left((\al-1) (N+U) + (S+U)\right)}^r(p_1^n)
\end{equation}
and the previous remarks in this section regarding the polynomial properties of $a_\la^r(\al)$ in $\al$.
%%%%%%%%%%%%%%%%%%%%%%%%%%%%%%%%%%%%%%%%%%%%%%%%%%%%%%%%%%%
\subsection{Proof of Theorems \ref{thm : zero} an \ref{thm : denes}}
We study further the combinatorial properties of the Laplace-Beltrami operator to prove Theorems \ref{thm : zero} an \ref{thm : denes}.\\
The power sum expansions of the actions of operators $N$, $U$ and $S$ on a given $p_\la$ have the following properties:
\begin{itemize}
\item $\ell(\mu) \neq \ell(\la) \Rightarrow [p_\mu]N(p_\la) = 0$
\item $\ell(\mu) \neq \ell(\la)-1 \Rightarrow [p_\mu]U(p_\la) = 0$
\item $\ell(\mu) \neq \ell(\la)+1 \Rightarrow [p_\mu]S(p_\la) = 0$
\end{itemize}
As a direct consequence, we have the lemma:
\begin{lem}
\label{lem : U}
For any integer partition $\la$ of $n$ with $\ell(\la) = n-r$ parts, the coefficient of $p_\la$ of ${D(\al)}^r(p_1^n)$ verifies
\begin{equation}
[p_\la]{D(\al)}^r(p_1^n) = \al^r[p_\la]U^r(p_1^n)
\end{equation}
Furthermore, if $\ell(\la) < n-r$, then $[p_\la]{D(\al)}^r(p_1^n) = 0$.
\end{lem}
The final remark of Lemma \ref{lem : U} proves Theorem \ref{thm : zero}. In order to prove Theorem~\ref{thm : denes}, it is sufficient to study operator $U$. Looking more precisely at the action of $U$ on a given power sum symmetric function $p_\la$, one can see that the terms $[p_\mu]U(p_\la)$ are non zero only when $\mu$ is of the form $\mu = (\la_1, \ldots, \la_{i-1}, \la_{i}+\la_{j}, \ldots, \la_{j-1},\la_{j+1},\ldots)$, i.e. when $\mu$ is obtained by combining (adding) two parts of $\la$. The contributions of the combinations of $\la_i$ and $\la_j$ to $[p_\mu]U(p_\la)$ are additive and equal to $\la_i\la_j$. The sum of all the contributions $\la_i\la_jp_\mu$ for all the possible $\mu$ combining two parts of $\la$ gives $U(p_\la)$. As an example, there are $n(n-1)/2$ ways of combining two parts of the partition $[1^n]$. Any of these combinations yields the partition $[1^{n-2}2^1]$ and in this particular case $\la_i\la_j = 1$. We have
\begin{equation}
U(p_1^n) = \frac{n(n-1)}{2}p_{[1^{n-2}2^1]}.
\end{equation}
Further, there are $(n-2)(n-3)/2$ ways of combining two $1$'s in $[1^{n-2}2^1]$ to get the partition $[1^{n-4}2^2]$ ($\la_i\la_j=1$) and $n-2$ ways of combining one $1$ and the part $2$ to get the partition $[1^{n-3}3^1]$ ($\la_i\la_j = 2$). It follows that
\begin{equation}
U(p_{[1^{n-2}2^1]}) = \frac{(n-2)(n-3)}{2}p_{[1^{n-4}2^2]} + 2(n-2)p_{[1^{n-3}3^1]}.
\end{equation}
Iterating $n-1$ times operator $U$ on $p_1^n$ yields a non zero coefficient only in $p_n$ (equal to $\al n!a^{n-1}_{[n^1]}$). This coefficient is the sum of the contributions of all the possible successive combinations of the parts of $[1^n]$. Figure \ref{fig : comb} shows how these various ways of combination contribute to the final coefficient in $p_n$ for $n=5$ ($U^4(p_{[1^5]}) = 3000p_{[5^1]}$).
%\vspace{-5mm}   
\begin{figure}[htbp]
  \begin{center}
    \includegraphics[width=0.75\textwidth]{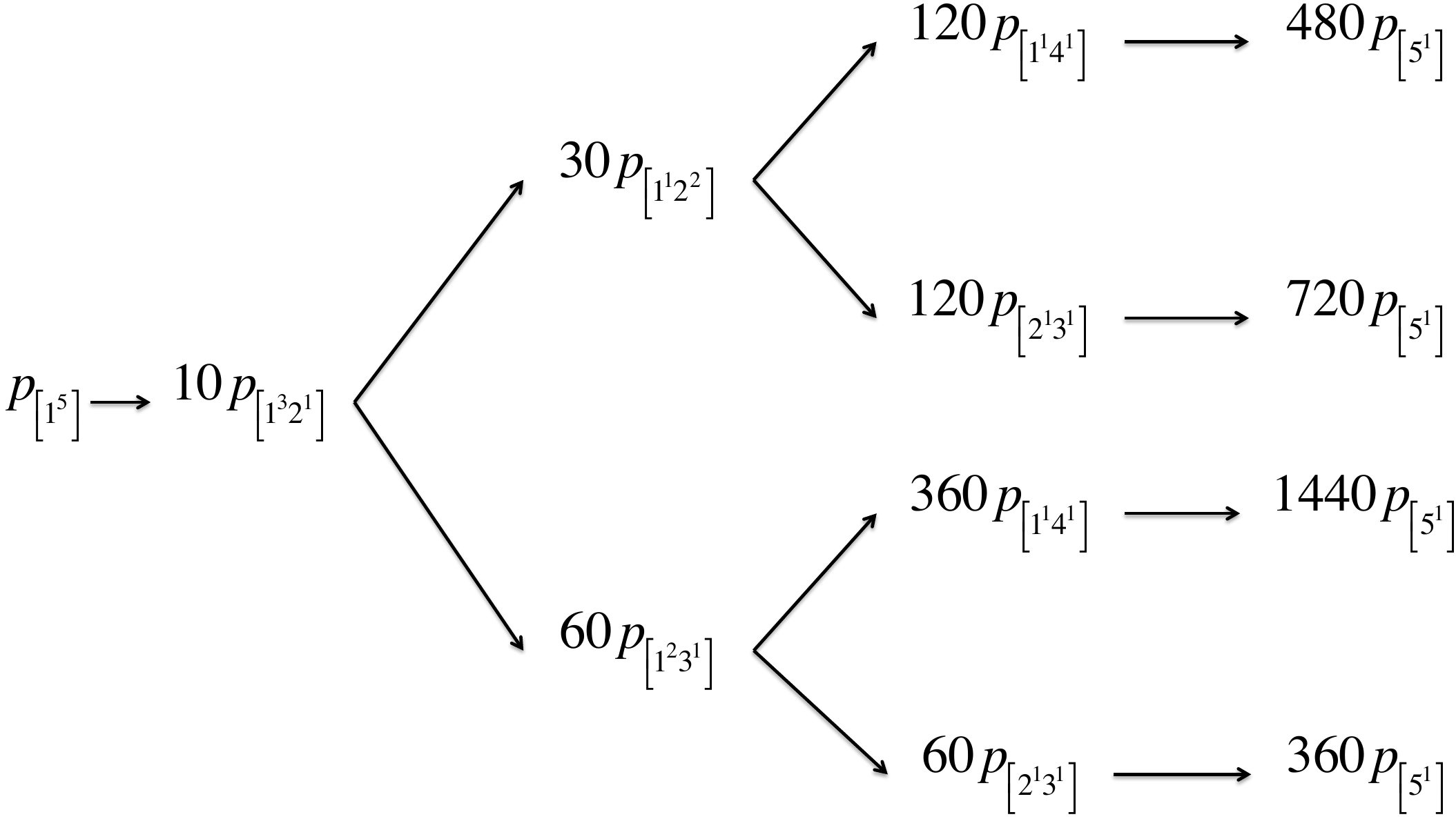}
    \caption{Illustration of operator U's action.}
    \label{fig : comb}
  \end{center}
\end{figure}
%The following lemma holds:
\begin{lem} \label{lem : recU} Let $i$ and $j$  be two integers greater or equal to 1. The following recursion holds for the coefficients in the power sum expansion of $U^r(p_1^n)$:
\begin{align}
\nonumber [p_{(i,j)}]U^{i+j-2}&(p_{[1^{i+j}]})=\\
&\frac{1}{2^{\delta_{ij}}}\binom{i+j}{i}\binom{i+j-2}{i-1}\left ([p_{i}]U^{i-1}(p_{[1^{i}]})\right ) \left ([p_{j}]U^{j-1}(p_{[1^{j}]})\right ) 
\end{align}
\end{lem}
\begin{proof}
As noticed in the previous paragraph, the application of operator $U^r$ on $p_{[1^{i+j}]}$ provides non zero coefficients in the $p_\la$'s where partition $\la$ is obtained by $r$ successive combinations of parts starting with $[1^{i+j}]$. At each iteration two parts are combined and the resulting coefficient depends only on the chosen parts (and not the whole partition at this step). As a result the coefficient obtained by $i+j-2$ combinations of parts to get partition $(i,j)$ from $[1^{i+j}]$ is equal to the product of the coefficients obtained by $k-1$ combinations of parts to get $(k)$ from $[1^{k}]$ for $k=i,j$ multiplied by:
\begin{itemize}
\item the number of ways of selecting the initial $1$-parts that will be combined to get $(i)$ (there are $\binom{i+j}{i}$ such ways)
\item the number of ways of selecting the $i-1$ (among $i+j-2$) iterations used to combine parts in order to get $(i)$ (there are $\binom{i+j-2}{i-1}$ such ways, all yielding obviously the same contribution).
\end{itemize}
When $i=j$ this total number is divided by $2$ not to count twice identical scenarios.   
\end{proof}
Combining the formula of Lemma \ref{lem : recU} with Equations (\ref{eq : arD}) and (\ref{eq : defU}) implies the following recursive formula on the number $a_{(n)}^{n-1}$:
\begin{equation}
\al n! a_{(n)}^{n-1}(\al) = \frac{1}{2}\sum_{i=1}^{n-1}i(n-i)\binom{n}{i}\binom{n-2}{i-1}\left (\al i!a_{(i)}^{i-1}(\al)\right)\left(\al (n-i)!a_{(n-i)}^{n-i-1}(\al)\right)
\end{equation}
which simplifies as
\begin{equation}\label{eq : recA}
\al n a_{(n)}^{n-1}(\al) = \frac{n}{2}\sum_{i=1}^{n-1}\binom{n-2}{i-1}\left (\al ia_{(i)}^{i-1}(\al)\right)\left(\al (n-i)a_{(n-i)}^{n-i-1}(\al)\right)
\end{equation}
Denote $t_n = \al n a_{(n)}^{n-1}(\al)$. As shown in \cite{Sh}, one can solve the above recursion by considering the generating function
\begin{equation}G(u) = \sum_{n \geq 1}\frac{t_n}{(n-1)!}u^n.\end{equation}
Using Equation (\ref{eq : recA}), one gets
\begin{equation}
\frac{1}{2}\frac{dG^2}{du}(u) = \frac{dG}{du}(u) - \frac{G(u)}{u}
\end{equation}
which gives
\begin{equation}
\frac{dG}{du} = \frac{d\ln(G)}{du} - \frac{d\ln}{du}.
\end{equation}
Using the initial conditions $G(0) = 0$ and $\frac{G(u)}{u}\mid_{u=0}=t_1=1$, one obtains the functional equation
\begin{equation} 
G(u) = u\exp(G(u)),
\end{equation}
and necessarily
\begin{equation}
G(u)=\sum_{n \geq 1}\frac{n^{n-2}}{(n-1)!}u^n.
\end{equation}
%As a result
%\begin{equation}
%a_{(n)}^{n-1}(\al) = \frac{1}{\al n}n^{n-2}
%\end{equation}
The proof of Theorem \ref{thm : denes} is completed by noticing that Lemma \ref{lem : recU} can be easily generalized as
\begin{equation}
Aut_\la[p_{\la}]U^{|\la|-\ell(\la)}(p_{[1^{|\la|}]}) =\binom{|\la|}{\la}\binom{|\la|-\ell(\la)}{\la-1}\prod_{i}\left ([p_{\la_i}]U^{\la_i-1}(p_{[1^{\la_i}]})\right )
\end{equation}
where $\la-1$ is the partition $(\la_1-1,\la_2-1,\ldots)$ of $|\la|-\ell(\la)$. We have:
\begin{align*}
\nonumber Aut_\la[p_{\la}]U^{|\la|-\ell(\la)}\left (\frac{p_{[1^{|\la|}]}}{\al^{|\la|}|\la|!}\right) &=\frac{(|\la|-\ell(\la))!}{\prod_i{(\la_i-1)!}}\prod_{i}\left ([p_{\la_i}]U^{\la_i-1}\left (\frac{p_{[1^{\la_i}]}}{\al^{\la_i}\la_i!}\right)\right )\\
\nonumber Aut_\la a^{|\la|-\ell(\la)}_{\la}(\al) &=\frac{(|\la|-\ell(\la))!}{\prod_i{(\la_i-1)!}}\prod_{i}a_{(\la_i)}^{\la_i-1}(\al)\\
Aut_\la a^{|\la|-\ell(\la)}_{\la}(\al) &=\frac{(|\la|-\ell(\la))!}{\prod_i{(\la_i-1)!}}\prod_{i}\frac{1}{\al \la_i}\la_i^{\la_i-2}
\end{align*}
\section{Further remarks}
\label{sec : FR}
\subsection{A generalization of Jackson's formula}
In this section we look at a less refined definition of Jack connection coefficients and focus only on the number of parts of the $\la^i$ . We define
\begin{equation}
\label{eq : pi}
 a_{n,p_1,p_2,\ldots,p_s}(\al) = \sum_{\la^i \vdash n,\; \ell(\la^i)=p_i} a_{\la^1,\la^2,\ldots,\la^s}(\al).
\end{equation}
Stanley \cite{S89} evaluated Jack symmetric functions at $x = I_k = (1,1,\ldots, 1,0,\ldots)$ ($k$ $1$'s) and showed  the classical formula
\begin{equation}
J^\al_\la(I_k) = R^\al_\la(k) := \prod_{s \in \la}(k+\al a'(s)-l'(s)).
\end{equation}
But as $p_\mu(I_k) = k^{\ell(\mu)}$ we have the following formula:
\begin{equation}
 \sum_{\mu \vdash n}\theta_\mu^\la k^{\ell(\mu)} = R^\al_\la(k). 
\end{equation}
Being true for any integer $k$, this polynomial identity  is actually true for any scalar $X$:
\begin{equation}
\label{eq : X}
 \sum_{\mu \vdash n}\theta_\mu^\la(\al) X^{\ell(\mu)} = R^\al_\la(X) = \prod_{s \in \la}(X+\al a'(s)-l'(s)). 
\end{equation}
As a result we get the following explicit general formulation:
\begin{thm}
\label{thm : genJackson}
 Let $a_{n, p_1,p_2,\ldots,p_s}(\al)$ be the coefficients defined in equation (\ref{eq : pi}) and $X_i$ ($1\leq i \leq s$) $s$ scalar indeterminate. We have the following formula for any integer $r \geq 1$:
\begin{equation}
\sum_{p_1,p_2,\ldots,p_s\geq 1} a_{n, p_1,p_2,\ldots,p_s}(\al)\prod_{1\leq i \leq s}X_i^{p_i} = \sum_{\beta \vdash n}\frac{1}{h_\beta(\alpha)h'_\beta(\alpha)}\prod_{1\leq i \leq s}R^\al_\beta(X_i)
\end{equation}
\end{thm}
Theorem \ref{thm : genJackson} is a generalization of the main formula in \cite{DMJ}.
\begin{rem} 
%Looking respectively at the coefficients of $X^n$, $X^{n-1}$ and $X$ in Equation \ref{eq : X}, one gets the following formulas for some particular Jack characters:
%\begin{equation}
%\theta_{[1^n]}^\la(\al) = 1,\;\;  \theta_{[1^{n-2}2^1]}^\la(\al) = \al n(\la') - n(\la),\;\; \theta_{[n^1]}^\la(\al) = \prod_{s \in \la \setminus\{(1,1)\}}(\al a'(s)-l'(s)),
%\end{equation}
%where $n(\la) = \sum_{i}(i-1)\la_i$.
In the case $\al = 1$ Jackson used this formula to compute explicitly $c^{(n)}_{\lambda^2,\ldots,\la^s}$ with $\la^i \in \{[n^1],[1^{n-2},2^1]\}$ for $2 \leq i \leq s$. In this particular case
\begin{align}
\theta_{[n^1]}^\la(1) &= \delta_{\la, [1^{n-t},t^1]}(-1)^{n-t-1}(n-t-1)!(t-1)!,\\
\theta_{[1^{n-2}2^1]}^{[1^{n-t},t^1]} &= -\frac{n(n-2t-1)}{2},\\
h_{[1^{n-t},t^1]}(1) &= (n-1)(t-1)!(n-t-1)!.
\end{align}
These very simple expressions allow a computation of the close form formula (\ref{eq : J}). When $\al \neq 1$ the formula for the Jack connection coefficients does not simplify.   
\end{rem}
\subsection{Differential Equations}
Consider the exponential generating function
\begin{equation}
F(\al,x,t) = \sum_{\la \vdash n, \; r\geq 0}a_\la^r(\al)p_\la(x)\frac{t^r}{r!}.
\end{equation} 
Using Theorem \ref{thm : ar}, one immediately gets
\begin{align}
 F(\al,x,t) = \sum_{r\geq 0}\left[ \frac{(tD(\al))^r}{r!}\left (\frac{p_1^n}{\al^nn!} \right )\right](x) = \left[\exp\left({tD(\al)}\right)\left (\frac{p_1^n}{\al^nn!} \right )\right](x)
\end{align}
As a consequence we have the following differential equation:
\begin{equation}
\frac{\partial F}{\partial t} - D(\al)F = 0.
\end{equation}
This equation is a general form of the differential equations studied by Goulden and Jackson for the generating series of the number of minimal transitive factorizations of a permutation into transpositions in \cite{GJ97}. 
\subsection{Unicellular hypermaps}
Using a similar development as in Section \ref{sec : PT} one can compute the coefficient $a_{\la,\mu,[1^{n-2}2^1]}(\al)$. We have
\begin{equation}
a_{\la,\mu,[1^{n-2}2^1]}(\al) = \frac{1}{\al^{\ell(\la)}z_\la}[p_\mu] D(\al)(p_\la).
\end{equation}
As a consequence $$a_{\la,\la,[1^{n-2}2^1]}(\al) = (\al-1)\al^{-\ell(\la)}z_\la^{-1}n(\la').$$
If $\mu = (\la_1, \ldots, \la_{i-1}, \la_{i}+\la_{j}, \ldots, \la_{j-1},\la_{j+1},\ldots)$ then $$a_{\la,\mu,[1^{n-2}2^1]}(\al) = \frac{\la_i\la_j}{2^{\delta_{i,j}}}\al^{-\ell(\la)+1}z_\la^{-1}.$$
If $\mu = (\la_1, \ldots, \la_{i-1}, \la_{i}-k,k,\ldots)$ then $$a_{\la,\mu,[1^{n-2}2^1]}(\al) = \frac{\la_i}{2^{\delta_{\la_i,2k}}}\al^{-\ell(\la)}z_\la^{-1}.$$
Unicellular hypermaps embedded in orientable (resp. locally orientable) surfaces are counted by $c^{(n)}_{\la,\mu}$ (resp. ${|B_n|^{-1}}b^{(n)}_{\la,\mu}$). Using the equation above and Lemma \ref{lem : sym}, we get:
\begin{align}
&c^{(n)}_{(n),[1^{n-2}2^1]} = 0, \;\; \frac{1}{|B_n|}b^{(n)}_{(n),[1^{n-2}2^1]} = \binom{n}{2},\\
&c^{(n)}_{(n-i,i),[1^{n-2}2^1]} = n/2^{\delta_{n,2i}},\;\;\frac{1}{|B_n|}b^{(n)}_{(n-i,i),[1^{n-2}2^1]} = n/2^{\delta_{n,2i}}
\end{align}
  
%At the end of the manuscript, right before the bibliography you might
%want to place an acknowledgement. This can be easily done by using the 
%command \verb!\acknowledgements! as you can see here.

%\bibliographystyle{abbrvnat}

% use the following instead if you encounter problems 
\bibliographystyle{alpha}
%\begin{thebibliography}{1}
%\bibliography{biblio}

\end{document}